\documentclass[11pt,a4paper]{article}

\usepackage{amsmath,amsthm,amssymb}
\usepackage[all]{xy}

\usepackage{colortbl}
\usepackage{arydshln}

\ADLnullwide

\makeatletter
\let\columncolor\relax
\def\endarray{%
   \adl@endarray \egroup \adl@arrayrestore \egroup
   \gdef\@preamble{}\CT@end}
\def\@array{%
   \adl@everyvbox\everyvbox
   \everyvbox{\adl@arrayinit \the\adl@everyvbox \everyvbox\adl@everyvbox}%
   \global\let\adl@hdashline@bgcolorsaved\adl@hdashline@bgcolor
   \global\let\adl@hdashline@bgcolor\@empty
   \adl@array}
\let\@@array\@array
\expandafter\def\expandafter\adl@arrayrestore\expandafter{%
   \adl@arrayrestore
   \global\let\adl@hdashline@bgcolor\adl@hdashline@bgcolorsaved}
\global\let\adl@hdashline@bgcolorsaved\@empty
\global\let\adl@hdashline@bgcolor\@empty
\def\adl@@colhtdp{%
   \ifdim\adl@height<\@tempdima \global\adl@height\@tempdima \fi
   \ifdim\adl@depth<\dp\z@  \global\adl@depth\dp\z@  \fi}
\adl@preaminit
\def\adl@ihdashline[#1/#2]{%
   \adl@hline\adl@connect\arrayrulewidth
   \ifnum0=`{\fi}%
   \adl@hdashline@bgcolor\crcr\noalign{\vskip-\arrayrulewidth}%
   \multispan{\adl@columns}\unskip \adl@hcline\z@[#1/#2]%
   \noalign{\ifnum0=`}\fi
   \futurelet\@tempa\adl@xhline}
\def\@classz{%
   \@classx
   \@tempcnta\count@
   \prepnext@tok
   \expandafter\CT@extract\the\toks\@tempcnta\columncolor!\@nil
   \@addtopreamble{%
      \setbox\z@\hbox\bgroup\bgroup
         \ifcase\@chnum
            \hskip\stretch{.5}\kern\z@
            \d@llarbegin \insert@column \d@llarend
            \hskip\stretch{.5}%
         \or \d@llarbegin \insert@column \d@llarend \hfill
         \or \hfill\kern\z@ \d@llarbegin \insert@column \d@llarend
         \or $\vcenter \@startpbox{\@nextchar}\insert@column \@endpbox$%
         \or \vtop \@startpbox{\@nextchar}\insert@column \@endpbox
         \or \vbox \@startpbox{\@nextchar}\insert@column \@endpbox
         \fi
      \egroup\egroup
      \begingroup
      \CT@setup
      \CT@column@color
      \CT@row@color
      \CT@do@color
      \endgroup
      \@tempdima\ht\z@
      \advance\@tempdima\minrowclearance
      \adl@colhtdp
      \vrule\@height\@tempdima\@width\z@
      \unhbox\z@}%
   \adl@setup@bgcolor
   \prepnext@tok}
\def\adl@setup@bgcolor{%
   \xdef\adl@hdashline@bgcolor{%
      \adl@hdashline@bgcolor
      \ifx\adl@hdashline@bgcolor\@empty\else &\fi
      \omit
      \setbox\z@\hbox{}\ht\z@\arrayrulewidth
      \begingroup
      \CT@setup
      \hskip\@tempdimb
      \CT@column@color
      \CT@do@color\unskip
      \hskip\@tempdimc
      \endgroup
      \box\z@}}%
\def\multicolumn#1#2#3{%
   \multispan{#1}%
   \begingroup\begingroup
   \def\adl@arrayrule{\adl@mcarrayrule{#1}}%
   \def\adl@arraydashrule{\adl@mcarraydashrule{#1}}%
   \def\adl@argarraydashrule{\adl@mcargarraydashrule{#1}}%
   \let\@addamp\adl@mcaddamp
   \let\adl@setup@bgcolor\relax
   \@mkpream{#2}\@addtopreamble\@empty
   \global\let\adl@preamble\@preamble\endgroup
   \let\@preamble\adl@preamble
   \def\@sharp{#3}\let\protect\relax
   \adl@activatepbox
   \adl@preaminit
   \let\CT@row@color\relax
   \let\CT@column@color\relax
   \let\CT@do@color\relax
   \@arstrut \@preamble\hbox{}\endgroup
   \global\advance\adl@currentcolumn#1\ignorespaces}
\makeatother

\newtheorem{theo}{Theorem}[section]

\newtheorem{lemm}[theo]{Lemma}
\newtheorem{coro}[theo]{Corollary}

\theoremstyle{definition}
\newtheorem{defi}[theo]{Definition}
\newtheorem{exam}[theo]{Example}

\newtheorem{rema}[theo]{Remark}

\newcommand{\Spec}{\mathop{\mathrm{Spec}}\nolimits}
\newcommand{\tdeg}{\mathop{\mathrm{tr.deg}}\nolimits}
\newcommand{\Mat}{\mathop{\mathrm{Mat}}\nolimits}
\newcommand{\GL}{\mathop{\mathrm{GL}}\nolimits}

\newcommand{\dep}{\mathop{\mathrm{dep}}\nolimits}
\newcommand{\Li}{\mathop{\mathrm{Li}}\nolimits}
\newcommand{\red}{\mathop{\mathrm{red}}\nolimits}
\newcommand{\wt}{\mathop{\mathrm{wt}}\nolimits}

\newcommand{\mbf}{\mathbf{f}}
\newcommand{\mbm}{\mathbf{m}}

\newcommand{\Z}{\mathbb{Z}}
\newcommand{\Q}{\mathbb{Q}}
\newcommand{\R}{\mathbb{R}}
\newcommand{\LL}{\mathbb{L}}
\newcommand{\T}{\mathbb{T}}
\newcommand{\E}{\mathbb{E}}
\newcommand{\Fq}{\mathbb{F}_{q}}
\newcommand{\Gm}{\mathop{\mathbb{G}_{m}}\nolimits}
\newcommand{\Ga}{\mathop{\mathbb{G}_{a}}\nolimits}

\newcommand{\Qb}{\overline{\mathbb{Q}}}

\newcommand{\Kinf}{K_{\infty}}
\newcommand{\Kinfb}{\overline{K_{\infty}}}
\newcommand{\Cinf}{\mathbb{C}_{\infty}}
\newcommand{\Kb}{\overline{K}}
\newcommand{\Kbt}{\overline{K}^{\times}}
\newcommand{\Fqtb}{\overline{\mathbb{F}_{q}(t)}}

\newcommand{\aub}{\underline{\alpha}}
\newcommand{\uval}{|\!|u|\!|_{\infty}}
\newcommand{\uival}{|\!|u_{i}|\!|_{\infty}}
\newcommand{\ujval}{|\!|u_{j}|\!|_{\infty}}
\newcommand{\uub}{\underline{u}}
\newcommand{\nub}{\underline{n}}

\newcommand{\zn}{\zeta(\nub)}

\newcommand{\vp}{\varphi}
\newcommand{\pit}{\widetilde{\pi}}
\newcommand{\Psit}{\widetilde{\Psi}}

\newcommand{\nq}{|\theta|_{\infty}^{\frac{n q}{q-1}}}
\newcommand{\niq}{|\theta|_{\infty}^{\frac{n_{i} q}{q-1}}}

\newcommand{\Ker}{\mathop{\mathrm{Ker}}\nolimits}

\title{On algebraic independence of certain multizeta values in characteristic $p$}
\author{Yoshinori Mishiba\thanks{
Graduate School of Mathematics, Kyushu University, 744, Motooka, Nishi-ku, Fukuoka, 819-0395, JAPAN.
\endgraf e-mail: \texttt{y-mishiba@math.kyushu-u.ac.jp}
\endgraf Supported by the JSPS Research Fellowships for Young Scientists.
}}
\date{}

\begin{document}
\maketitle

\begin{abstract}
In this paper, we study multizeta values over function fields in characteristic $p$.
For each $d \geq 2$, we show that when the constant field has cardinality $> 2$,
the field generated by all multizeta values of depth $d$ is of infinite transcendence degree over the field
generated by all single zeta values.
As a special case, this gives an affirmative answer to the function field analogue of a question of Y.\ Andr\'e.
\end{abstract}

\section{Introduction}
\subsection{Classical case}
The multiple zeta values (MZVs) in characteristic 0 was defined by Euler (depth two) and Hoffman (higher depth).
For a $d$-tuple of positive integers $\nub = (n_{1}, \dots, n_{d}) \in (\Z_{\geq 1})^{d}$ with $n_{1} \geq 2$, it is defined by
$$\zeta_{\Z}(\nub) = \zeta_{\Z}(n_{1}, \dots, n_{d}) := \sum_{m_{1} > \cdots > m_{d} \geq 1} \frac{1}{m_{1}^{n_{1}} \cdots m_{d}^{n_{d}}} \in \R^{\times}.$$
The sum $\wt(\nub) := \sum_{i} n_{i}$ is called the weight and $\dep(\nub) := d$ is called the depth of $\zeta_{\Z}(\nub)$.
One of the goals of this topic is to determine all algebraic relations over $\Qb$ among the MZVs.
Although many relations among MZVs are known, very few linear/algebraic independence results on MZVs are known.
For example, Euler proved that when $d = 1$, the ratio $\zeta_{\Z}(n) / (2 \pi \sqrt{-1})^{n}$ is a rational number if and only if $n \geq 2$ is an even integer.
However, we do not know whether $\zeta_{\Z}(n) / \pi^{n}$ is a transcendental number for each odd integer $n \geq 3$.
It is conjectured that $\pi, \zeta_{\Z}(3), \zeta_{\Z}(5), \zeta_{\Z}(7), \dots$ are algebraically independent over $\Qb$.
For the higher depth case, Goncharov (\cite{Gonc}) conjectured that MZVs of different weights are linearly independent over $\Q$.
Moreover, it is considered that this conjecture is true even if we replace $\Q$ by $\Qb$.
Andr\'e (\cite[p. 231]{Andr}) asked whether there exists $\nub$ such that $\zeta_{\Z}(\nub) \not\in \Q[\zeta_{\Z}(2), \zeta_{\Z}(3), \zeta_{\Z}(4), \zeta_{\Z}(5), \dots]$.
In this paper, a consequence of our main result is to give an affirmative answer to the function field analogue of this question.

\subsection{Positive characteristic multizeta values}
Let $K := \Fq(\theta)$ be the rational function field over the finite field of $q$ elements with variable $\theta$,
$p$ the characteristic of $K$,
$\Kinf := \Fq(\!(\theta^{-1})\!)$ the $\infty$-adic completion of $K$,
$\Kinfb$ a fixed algebraic closure of $\Kinf$,
$\Cinf$ the $\infty$-adic completion of $\Kinfb$,
and $\Kb$ the algebraic closure of $K$ in $\Cinf$.
We fix a ($q-1$)-st root of $-\theta$ and let
$$\pit := (-\theta)^{\frac{q}{q-1}} \prod_{i=1}^{\infty}\left(1-\theta^{1-q^i}\right)^{-1} \in (-\theta)^{\frac{1}{q-1}} \cdot \Kinf^{\times}$$
be the fundamental period of the Carlitz module.
This is a generator of the kernel of the exponential map of the Carlitz module
and a function field analogue of $2 \pi \sqrt{-1}$ which is a generator of the kernel of the usual exponential map.
As $\# \Fq[\theta]^{\times} = q-1$, we say that an integer $n$ is ``odd'' if $q-1$ does not divide $n$, and ``even'' if $q-1$ divides $n$.

In this paper, an \textit{index} means an element of $(\Z_{\geq 1})^{d}$ for some positive integer $d \geq 1$.
Thakur (\cite{Tha1}) defined the {\it positive characteristic multizeta values} (also denoted by MZVs) by
$$\zn = \zeta(n_{1}, \dots, n_{d}) := \sum \frac{1}{a_{1}^{n_{1}} \cdots a_{d}^{n_{d}}} \in \Kinf^{\times}$$
for indices $\nub = (n_{1}, \dots, n_{d})$, where the sum is over all monic polynomials $a_{i}$ in $\Fq[\theta]$ such that
$\deg a_{1} > \cdots > \deg a_{d} \geq 0$.
Thakur (\cite{Tha2}) showed that MZVs are non-zero.
Anderson and Thakur (\cite{AT09}) showed that the MZVs have an interpretation as periods of certain $t$-motives,
which are iterated extensions of the trivial $t$-motive by Carlitz motives.
It is clear that $\zeta(p^{e} \nub) = \zn^{p^{e}}$ for all $e \geq 0$, where we set $p^{e} \nub := (p^{e} n_{1}, \dots, p^{e} n_{d})$.
Carlitz (\cite{Carl}) showed that $\zeta(n) / \pit^{n} \in K^{\times}$ if and only if $n \geq 1$ is an ``even'' integer.
This is an analogue of the relations of the special zeta values at even integers.
Yu (\cite{Yu1}) proved that $\zeta(n) / \pit^{n} \not\in \Kb$ for each ``odd'' integer $n \geq 1$.
Moreover, Chang and Yu (\cite{ChYu}) proved that the elements $\pit, \zeta(n_{1}), \dots, \zeta(n_{d})$ are algebraically independent over $\Kb$
if $n_{i}$ is ``odd'' for each $i$ and $n_{i} / n_{j}$ is not an integral power of $p$ for each $i \neq j$.
Chang, Papanikolas and Yu (\cite{CPY}) also showed the algebraic independence of MZVs when the constant field $\Fq$ varies.
Several results on the higher depth case were also proved.
Chang (\cite{Cha2}) showed that MZVs of different weights are linearly independent over $\Kb$.
In \cite{Mish}, we proved that $\pit$, $\zeta(n)$ and $\zeta(n,n)$ are algebraically independent over $\Kb$ if $2n$ is ``odd''.
In this paper, we prove the following theorem:
\begin{theo}\label{main_zeta}
Let $d \geq 1$ be a positive integer, and let $n_{1}, \dots, n_{d} \geq 1$ be $d$ distinct positive integers.
If $n_{i}$ is not divisible by $q-1$ for each $i$ and $n_{i} / n_{j}$ is not an integral power of $p$ for each $i \neq j$,
then the the following $1 + \frac{d(d+1)}{2}$ elements
$$
\begin{array}{c}
\pit, \ \ \ \ \zeta(n_{1}), \ \ \ \ \zeta(n_{2}), \ \ \ \ \zeta(n_{3}), \ \ \ \ \zeta(n_{4}), \ \dots \dots \dots, \ \zeta(n_{d}), \ \ \ \ \ \ \ \ \\
\zeta(n_{1}, n_{2}), \ \ \zeta(n_{2}, n_{3}), \ \ \zeta(n_{3}, n_{4}), \ \dots \dots, \ \zeta(n_{d-1}, n_{d}), \ \\
\zeta(n_{1}, n_{2}, n_{3}), \zeta(n_{2}, n_{3}, n_{4}), \dots, \zeta(n_{d-2}, n_{d-1}, n_{d}), \\
\vdots \\
\zeta(n_{1}, n_{2}, \dots, n_{d-1}), \zeta(n_{2}, n_{3}, \dots, n_{d}), \\
\zeta(n_{1}, n_{2}, \dots, n_{d})
\end{array}$$
are algebraically independent over $\Kb$.
\end{theo}
Theorem \ref{main_zeta} provides many MZVs which are algebraically independent over $\Kb$.
The next theorem gives a positive answer to the function field analogue of a question in \cite[p. 231]{Andr}.
\begin{theo}
For each positive integer $d \geq 1$, we set $K_{d}$ to be the field generated by the MZVs of depth 1 or $d$ over $K$.
When $q \neq 2$, we have
$$\tdeg_{K_{1}} K_{d} = \infty$$
for each $d \geq 2$.
\end{theo}
\begin{proof}
Since $q \neq 2$, the set $\Z_{\geq 1} \smallsetminus ((q-1) \Z_{\geq 1} \cup p \Z_{\geq 1})$ is an infinite set.
We denote the elements of this set by $n_{1}, n_{2}, n_{3}, \dots$.
Hence we have $K_{1} = K(\zeta(n_{1}), \zeta(n_{2}), \zeta(n_{3}), \dots)$.
By Theorem \ref{main_zeta}, the elements
$\zeta(n_{1}, \dots, n_{d}), \zeta(n_{d + 1}, \dots, n_{2 d}), \zeta(n_{2 d + 1}, \dots, n_{3 d}), \dots$
are algebraically independent over $K_{1}$.
\end{proof}

\begin{rema}
$(1)$ Similarly, we can prove that for any integers $d_{1}, d_{2}, d_{3}, \dots \geq 2$,
there exist indices $\nub_{1}, \nub_{2}, \nub_{3}, \dots$ such that
$\dep(\nub_{j}) = d_{j}$ for each $j$ and $\zeta(\nub_{1}), \zeta(\nub_{2}), \zeta(\nub_{3}), \dots$ are algebraically independent over $K_{1}$.

$(2)$ When $q = 2$, Chang (\cite{Cha2}) showed that either $\zeta(1, 2)$ or $\zeta(2, 1)$ is transcendental over $K_{1}$.
However we do not know whether there exist infinitely many MZVs which are algebraically independent over $K_{1}$ when $q = 2$.
\end{rema}

By Theorem \ref{main_zeta}, we may obtain some lower bounds of the dimension of the vector space over $K$ (or $\Kb$)
spanned by the MZVs of fixed weight.
We do not pursue this problem in this paper and content ourselves with stating the following lower bound
of the transcendental degree of the field generated by the MZVs of bounded weights,
which is easily obtained from Theorem \ref{main_zeta}:
\begin{coro}
Let $w \geq 1$ be a positive integer.
If there exist positive integers $d_{1}, \dots, d_{r} \geq 1$
and an ``odd'' positive integer $n_{i j} \geq 1$ for each $1 \leq i \leq r$ and $1 \leq j \leq d_{i}$
such that $n_{i j} / n_{i' j'}$ is not an integral power of $p$ for each $(i, j) \neq (i', j')$
and $\sum_{j} n_{i j} \leq w$ for each $i$,
then we have
$$\tdeg_{\Kb} \Kb(\pit, \zeta(\nub) | \wt(\nub) \leq w) \geq 1 + \sum_{i=1}^{r} \frac{d_{i} (d_{i} + 1)}{2}.$$
\end{coro}

\subsection{Carlitz multiple polylogarithms}
In \cite{Cha2}, Chang defined the {\it Carlitz multiple polylogarithms} (CMPLs) by
$$\Li_{\nub}(z_{1}, \dots, z_{d}) := \sum_{i_{1} > \cdots > i_{d} \geq 0}
\frac{z_{1}^{q^{i_{1}}} \cdots z_{d}^{q^{i_{d}}}}
{((\theta-\theta^{q}) \cdots (\theta-\theta^{q^{i_{1}}}))^{n_{1}} \cdots ((\theta-\theta^{q}) \cdots (\theta-\theta^{q^{i_{d}}}))^{n_{d}}}$$
for indices $\nub$.
It converges if $|z_{i}|_{\infty} < \niq$ for each $i$, where $|-|_{\infty}$ is an $\infty$-adic valuation on $\Cinf$.
The weight and depth of (values of) CMPLs are also defined to be $\wt(\nub)$ and $\dep(\nub)$.
In \cite{AT90}, Anderson and Thakur showed that $\zeta(n)$ is described as a $K$-linear combination of the values of CMPLs
of weight $n$ and depth one at rational points for each $n \geq 1$.
Moreover, in \cite{Cha2}, Chang showed that for each index $\nub$ with $\wt(\nub) = w$ and $\dep(\nub) = d$,
$\zeta(\nub)$ is described as a $K$-linear combination of the values of CMPLs of weight $w$ and depth $d$ at rational points.
He also proved that CMPLs take non-zero values when $z_{i} \neq 0$ for each $i$.
We are interested in the algebraic independence of their values at algebraic points over $\Kb$.

Let $n \geq 1$ be a positive integer, and let $\alpha_{1}, \dots, \alpha_{r} \in \Kbt$ be algebraic points such that $|\alpha_{j}|_{\infty} < \nq$ for each $j$.
Papanikolas (\cite{Papa}), Chang and Yu (\cite{ChYu}) proved that
if $\pit^{n}, \Li_{n}(\alpha_{1}), \dots, \Li_{n}(\alpha_{r})$ are linearly independent over $K$,
then they are algebraically independent over $\Kb$.
Let $n_{1}, \dots, n_{d} \geq 1$ be positive integers such that $n_{i} / n_{j}$ is not an integral power of $p$ for each $i \neq j$.
For each $i$, let $\alpha_{i 1}, \dots, \alpha_{i r_{i}} \in \Kbt$ be algebraic points such that $|\alpha_{i j}|_{\infty} < \niq$ for each $j$.
Chang and Yu (\cite{ChYu}) also proved that
if $\pit^{n_{i}}, \Li_{n_{i}}(\alpha_{i 1}), \dots, \Li_{n_{i}}(\alpha_{i r_{i}})$ are linearly independent over $K$ for each $i$,
then the elements of $\{ \pit \} \cup \{ \Li_{n_{i}}(\alpha_{i j}) | i, j \}$ are algebraically independent over $\Kb$.
As in the case of the MZVs, several results on the higher depth case were also proved.
Chang (\cite{Cha2}) showed that values of Carlitz multiple polylogarithms at algebraic points of different weights are linearly independent over $\Kb$.
In \cite{Mish}, we proved that $\pit$, $\Li_{n}(\alpha)$ and $\Li_{n,n}(\alpha,\alpha)$ are algebraically independent over $\Kb$
if $2n$ is ``odd'' and $\alpha \in K^{\times}$ with $|\alpha|_{\infty} < \nq$.
In this paper, we prove the following theorem:
\begin{theo}\label{main_li}
Let $d \geq 1$ be a positive integer, and let $n_{1}, \dots, n_{d} \geq 1$ be $d$ distinct positive integers.
For each $i$, we take a rational point $\alpha_{i} \in K^{\times}$ such that $|\alpha_{i}|_{\infty} < \niq$.
If $n_{i}$ is not divisible by $q-1$ for each $i$ and $n_{i} / n_{j}$ is not an integral power of $p$ for each $i \neq j$,
then the cardinality of the set $\{ \pit \} \cup \{ \Li_{n_{\ell}, n_{\ell+1}, \dots, n_{k}}(\alpha_{\ell}, \alpha_{\ell+1}, \dots, \alpha_{k}) | 1 \leq \ell \leq k \leq d \}$
is $1 + \frac{d(d+1)}{2}$ and all elements of this set are algebraically independent over $\Kb$.
\end{theo}

\subsection{Outline of this paper}
In Section \ref{notat}, we define notations which are used in this paper.
In Section \ref{t-mot}, at first we review the (pre-)$t$-motives which were originally defined by Anderson (\cite{Ande}).
We explain the way how we obtain periods from pre-$t$-motives following the work of Anderson and Thakur (\cite{AT90}, \cite{AT09}).
Then we recall Papanikolas' theory (\cite{Papa}) which states that the transcendental degree of the field generated by periods in question
over a base field coincides with the dimension of the ``motivic Galois group'' of a pre-$t$-motive.
As an example (see Example \ref{ex-L}), we see that MZVs and CMPLs at algebraic points appear as periods of some pre-$t$-motives.
The primary tools of proving the main results are to apply Papanikolas' theory.
In Section \ref{proof}, we give proofs of Theorems \ref{main_zeta} and \ref{main_li}.
These are simultaneously proved as corollaries of Theorem \ref{main}.
This theorem is proved by using the arguments of Section \ref{t-mot}.
In \ref{app_proof}, we give a proof of Theorem \ref{alg-dep1} which states a criterion
of the algebraic independence of ``depth one elements''.
This gives a generalization of the main result of \cite{ChYu}.

\section{Notations}\label{notat}
We continue to use the notations of the previous section.
Let $t$ be a variable independent of $\theta$.
Let $\T := \{f \in \Cinf[\![t]\!] | f \mathrm{ \ converges \ on \ } |t|_{\infty} \leq 1\}$ be the Tate algebra
and $\LL$ the fractional field of $\T$.
We set
$$\E := \left\{\sum a_{i} t^{i} \in \Cinf[\![t]\!] \Bigl| \lim_{i \to \infty} \sqrt[i]{|a_{i}|_{\infty}} = 0, \ [\Kinf(a_{0}, a_{1}, \dots):\Kinf] < \infty \right\}.$$
For any integer $n \in \Z$ and any formal Laurent series $f = \sum_{i} a_{i} t^{i} \in \Cinf(\!(t)\!)$, let
$$ f^{(n)} := \sum_{i} a_{i}^{q^n} t^{i} $$
be the $n$-fold twist of $f$, and set $\sigma(f) := f^{(-1)}$.
The fields $\LL$ and $\Kb(t)$ are stable under the operation $f \mapsto f^{(n)}$
and we have $\LL^{\sigma = 1} = \Fq(t)$ where $(-)^{\sigma = 1}$ is the $\sigma$-fixed part.

\begin{defi}\label{def_index}
Let $d \geq 1$ be a positive integer.
We set
$$I_{d} := \{ (i, j) \in \Z^{2} | 1 \leq j < i \leq d + 1 \}.$$
We define a depth of $(i, j) \in I_{d}$ by $\dep(i, j) := i - j$
and a total order on $I_{d}$ by setting $(i, j) \leq (k, \ell)$ if either $\dep(i, j) = \dep(k, \ell)$ and $j \leq \ell$ (hence $i \leq k$), or $\dep(i, j) < \dep(k, \ell)$.
The order on $I_{d}$ is illustrated as the following diagram:
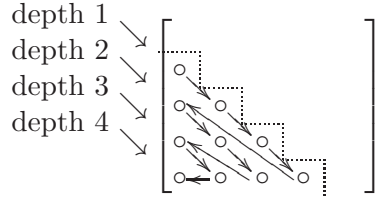
\begin{eqnarray*}
\begin{matrix}
\mathrm{depth} \ 1 \\
\mathrm{depth} \ 2 \\
\mathrm{depth} \ 3 \\
\mathrm{depth} \ 4 \\
\\ \\
\end{matrix}
\
\begin{matrix}
\searrow \\
\searrow \\
\searrow \\
\searrow \\
\\
\end{matrix}
\begin{bmatrix}
& & & & & \\
\cdashline{1-1}
\multicolumn{1}{c:}{\circ \makebox[0pt][l]{\smash{\begin{xy}\ar (3, -3)\end{xy}}}} & & & & & \\
\cdashline{2-2}
\circ \makebox[0pt][l]{\smash{\begin{xy}\ar (3, -3)\end{xy}}} & \multicolumn{1}{c:}{\circ \makebox[0pt][l]{\smash{\begin{xy}\ar (3, -3)\end{xy}}}} & & & & \\
\cdashline{3-3}
\circ \makebox[0pt][l]{\smash{\begin{xy}\ar (3, -3)\end{xy}}} & \circ \makebox[0pt][l]{\smash{\begin{xy}\ar (3, -3)\end{xy}}} &
\multicolumn{1}{c:}{\circ \makebox[0pt][l]{\smash{\begin{xy}\ar (3, -3)\end{xy}}}} & & & \\
\cdashline{4-4}
\circ \makebox[0pt][l]{\raisebox{2pt}[0pt][0pt]{\begin{xy}\ar (-3, 0)\end{xy}}} & \circ \makebox[0pt]{\raisebox{2pt}[0pt][0pt]{\begin{xy}\ar (-8, 5)\end{xy} \ }} &
\circ \makebox[0pt]{\raisebox{2pt}[0pt][0pt]{\begin{xy}\ar (-14, 10)\end{xy} \ \ \ \ \ }} & \multicolumn{1}{c:}{\circ} & &
\end{bmatrix}
\end{eqnarray*}
For each $(i, j) \in I_{d}$ and a $d$-tuple of symbol $\underline{y} = (y_{1}, \dots, y_{d})$, we set
$$\underline{y}_{i j} := (y_{j}, y_{j+1}, \dots, y_{i-1}).$$
\end{defi}

Note that the MZV $\zeta(n_{1}, \dots, n_{d})$ appears as a period of a $t$-motive of dimension $d+1$ (Example \ref{ex-L}).
Moreover, the MZV $\zeta(n_{j}, \dots, n_{i-1})$ for $(i, j) \in I_{d}$ appears as an $(i, j)$-th component of a matrix of periods of that $t$-motive.

We set
$$\Omega(t) := (-\theta)^{-\frac{q}{q-1}} \prod_{i=1}^{\infty}\left(1-\frac{t}{\theta^{q^i}}\right) \in \Kinfb[\![t]\!]$$
which is an element of $\E$.
Since $\Omega$ has a simple zero at $\theta^{q^{i}}$ for each $i = 1, 2, \dots$, it is transcendental over $\Kb(t)$.
It satisfies the equation
$$\Omega^{(-1)} = (t-\theta) \Omega$$
and we have
$$\Omega(\theta) = \frac{1}{\pit}.$$

We set $D_{0} := 1$ and $D_{i} := \prod_{j=0}^{i-1} (\theta^{q^{i}} - \theta^{q^{j}})$ for $i \geq 1$.
For each integer $n \geq 0$ with the $q$-expansion $n = \sum_{i} n_{i} q^{i}$ ($0 \leq n_{i} < q$),
the Carlitz factorial is defined by
$$\Gamma_{n+1} := \prod_{i} D_{i}^{n_{i}}.$$
Let $\nub = (n_{1}, \dots, n_{d})$ be an index
and $\uub = (u_{1}, \dots, u_{d}) \in (\Kb[t])^{d}$ a $d$-tuple of polynomials.
For a polynomial $u = \sum_{j} \alpha_{j} t^{j} \in \Kb[t]$, we set $\uval := \max_{j} |\alpha_{j}|_{\infty}$.
When $\uival < \niq$ for each $i$, we set
$$L_{\uub,\nub}(t) := \sum_{i_{1} > \cdots > i_{d} \geq 0} \frac{u_{1}^{(i_{1})} \cdots u_{d}^{(i_{d})}}
{((t-\theta^{q}) \cdots (t-\theta)^{q^{i_{1}}})^{n_{1}} \cdots ((t-\theta^{q}) \cdots (t-\theta)^{q^{i_{d}}})^{n_{d}}} \in \Kinfb[\![t]\!],$$
which converges on $|t|_{\infty} < |\theta|_{\infty}^{q}$ and satisfies the equation
$$L_{\uub,\nub}^{(-1)} = \frac{u_{d}^{(-1)}}{(t-\theta)^{n_{1} + \cdots + n_{d-1}}} L_{\uub_{d 1}, \nub_{d 1}}
+ \frac{L_{\uub,\nub}}{(t-\theta)^{n_{1} + \cdots + n_{d}}},$$
where we set $L_{\uub_{1 1}, \nub_{1 1}} = L_{\emptyset, \emptyset} := 1$.
When $\uub = \aub \in \Kb^{d}$ with $|\alpha_{i}|_{\infty} < \niq$ for each $i$, we have
$L_{\aub, \nub}(\theta) = \Li_{\nub}(\aub)$.
Anderson and Thakur (\cite{AT90}, \cite{AT09}) showed that there exists a polynomial $H_{n-1} \in \Fq[\theta, t]$ for each $n \geq 1$ such that
$|\!|H_{n-1}|\!|_{\infty} < \nq$ and
$L_{H(\nub), \nub}(\theta) = \Gamma_{n_{1}} \cdots \Gamma_{n_{d}} \zeta(\nub)$
where $H(\nub) := (H_{n_{1}-1}, \dots, H_{n_{d}-1})$.

\section{Review of pre-$t$-motives}\label{t-mot}
In this section, we review the notions of pre-$t$-motives and Papanikolas' theory for pre-$t$-motives.
For more details, see \cite{Papa}.
A {\it pre-$t$-motive} is an \'etale $\vp$-module over $(\Kb(t), \sigma)$;
this means a finite-dimensional $\Kb(t)$-vector space $M$ equipped with a $\sigma$-semilinear bijective map
$\vp \colon M \to M$.
A morphism of pre-$t$-motives is a $\Kb(t)$-linear map which is compatible with the $\vp$'s.
A tensor product of two pre-$t$-motives are defined naturally.
For any pre-$t$-motive $M$, the {\it Betti realization} of $M$ is defined by
$$M^{B} := \left(\LL \otimes_{\Kb(t)} M \right)^{\sigma \otimes \vp = 1},$$
where $(-)^{\sigma \otimes \vp = 1}$ is the $\sigma \otimes \vp$-fixed part.
A pre-$t$-motive $M$ is called {\it rigid analytically trivial} if the natural injection
$\LL \otimes_{\Fq(t)} M^{B} \hookrightarrow \LL \otimes_{\Kb(t)} M$
is an isomorphism.
The category of rigid analytically trivial pre-$t$-motives forms a neutral Tannakian category over $\Fq(t)$
with fiber functor $M \mapsto M^{B}$.
For any such $M$, we denote by $G_{M}$ the fundamental group of the Tannakian subcategory generated by $M$ with respect to the Betti realization.
By definition, $G_{M}$ is an $\Fq(t)$-subgroup scheme of $\GL(M^{B})$.

Let $\Phi \in \GL_{r}(\Kb(t))$ be a matrix.
We consider the system of Frobenius difference equations
\begin{eqnarray}\label{Fr-eq}
\Psi^{(-1)} = \Phi \Psi
\end{eqnarray}
with solution entries of $\Psi = (\Psi_{i j})$ in $\LL$.
The matrix $\Phi$ defines the pre-$t$-motive
$M_{\Phi} := \Kb(t)^{r}$ with $$\vp(x_{1}, \dots, x_{r}) = (x_{1}^{(-1)}, \dots, x_{r}^{(-1)}) \Phi.$$
The pre-$t$-motive $M_{\Phi}$ is rigid analytically trivial if and only if
the system of Frobenius difference equations (\ref{Fr-eq}) has a solution matrix $\Psi$ in $\GL_{r}(\LL)$,
and in this case $\Psi^{-1} \mbm$ forms an $\Fq(t)$-basis of $(M_{\Phi})^{B}$,
where $\mbm$ is the standard basis of $\Kb(t)^{r}$
on which the action of $\vp$ is presented as $\Phi$.
Such matrix $\Psi$ is called a {\it rigid analytic trivialization} of $\Phi$,
and the values $\Psi_{i j}(\theta)$ of its components at $t = \theta$ (if they converge) are called {\it periods} of $M_{\Phi}$.
For such $\Psi$, we set $\Psit := \Psi_{1}^{-1} \Psi_{2} \in \GL_{r}(\LL \otimes_{\Kb(t)} \LL)$,
where $\Psi_{1}$ (resp. $\Psi_{2}$) is the matrix in $\GL_{r}(\LL \otimes_{\Kb(t)} \LL)$ such that $(\Psi_{1})_{ij} = \Psi_{ij} \otimes 1$ (resp. $(\Psi_{2})_{ij} = 1 \otimes \Psi_{ij}$).
Let $X = (X_{ij})$ be the $r \times r$ matrix of independent variables $X_{ij}$.
We define an $\Fq(t)$-algebra homomorphism $\nu$ by
$$\nu \colon \Fq(t)[X, 1/\det X] \to \LL \otimes_{\Kb(t)} \LL ; \ X_{ij} \mapsto \Psit_{ij}$$
and set
$$G_{\Psi} := \Spec (\Fq(t)[X, 1/\det X] / \ker \nu) \subset \GL_{r / \Fq(t)}.$$
For each $\Fq(t)$-algebra $R$, we have the map given by
\begin{eqnarray}\label{map_G}
G_{\Psi}(R) \to G_{M_{\Phi}}(R); \ g \mapsto (\mbf \cdot \Psi^{-1} \mbm \mapsto \mbf g^{-1} \cdot \Psi^{-1} \mbm)
\end{eqnarray}
where $\mbf$ runs over all elements of $\Mat_{1 \times r}(R)$.

\begin{theo}[{\cite[Theorems 4.3.1, 4.5.10, 5.2.2]{Papa}}]\label{Psi-M}
Let $\Phi$ and $\Psi$ be matrices satisfying (\ref{Fr-eq}),
and let $G_{M_{\Phi}}$ and $G_{\Psi}$ be as above.

$(1)$ The scheme $G_{\Psi}$ is a smooth subgroup scheme of $\GL_{r / \Fq(t)}$ and 
the above map $G_{\Psi} \to G_{M_{\Phi}}$ is an isomorphism of group schemes over $\Fq(t)$.

$(2)$ Let $\Kb(t)(\Psi)$ be the field generated by the entries of $\Psi$ over $\Kb(t)$.
Then we have
$$\dim G_{\Psi} = \tdeg_{\Kb(t)} \Kb(t)(\Psi).$$

$(3)$ Assume that $\Phi \in \Mat_{r}(\Kb[t])$, $\Psi \in \GL_{r}(\T) \cap \Mat_{r}(\E)$,
and $\det \Phi = c (t-\theta)^{d}$ for some $c \in \Kbt$ and $d \geq 0$.
Let $\Kb(\Psi(\theta))$ be the field generated by the entries of $\Psi(\theta)$ over $\Kb$.
Then we have
$$\tdeg_{\Kb(t)} \Kb(t)(\Psi) = \tdeg_{\Kb} \Kb(\Psi(\theta)).$$
\end{theo}

\begin{rema}\label{rem_ABP}
The result $(3)$ in Theorem \ref{Psi-M} is rooted in the deep result in \cite{ABP},
which is addressed as ABP-criterion.
However, the restriction of the condition on $\det \Phi$ originated from Anderson $t$-motives
but such restriction indeed can be relaxed (see \cite{Cha1}).
But for our purpose, the above is sufficient and so we do not state the refined version given in \cite{Cha1}.
\end{rema}

\begin{exam}\label{ex-C}
The Carlitz pre-$t$-motive $C$ is the pre-$t$-motive defined by the $1 \times 1$-matrix $\begin{bmatrix} t-\theta \end{bmatrix}$.
Since $\Omega^{(-1)} = (t-\theta) \Omega$, the Carlitz pre-$t$-motive is rigid analytically trivial.
Since $\Omega$ is transcendental over $\Kb(t)$, we have $\dim G_{[\Omega]} = 1$, and thus $G_{C} = G_{[\Omega]} = \Gm$.
\end{exam}

\begin{exam}\label{ex-L}
Let $\nub = (n_{1}, \dots, n_{d})$ be an index and $\uub = (u_{1}, \dots, u_{d}) \in (\Kb[t])^{d}$ be a $d$-tuple of polynomials such that
$\uival < \niq$ for each $i$.
We consider $(d+1) \times (d+1)$-matrices
$$\Phi :=
\begin{bmatrix}
(t-\theta)^{n_{1} + \cdots + n_{d}}                & 0                                                           & 0         & \cdots                            & 0 \\
u_{1}^{(-1)} (t-\theta)^{n_{1} + \cdots + n_{d}} & (t-\theta)^{n_{2} + \cdots + n_{d}}                & 0         & \cdots                            & 0 \\
0                                                           & u_{2}^{(-1)} (t-\theta)^{n_{2} + \cdots + n_{d}} & \ddots &                                       & \vdots \\
\vdots                                                    &                                                             & \ddots & (t-\theta)^{n_{d}}                & 0 \\
0                                                           & \cdots                                                   & 0         & u_{d}^{(-1)} (t-\theta)^{n_{d}} & 1 \\ \end{bmatrix}$$
and
$$\Psi :=
\begin{bmatrix}
\Omega^{n_{1} + \cdots + n_{d}}                                         & 0                                                                                  & 0         & \cdots            & 0 \\
\Omega^{n_{1} + \cdots + n_{d}} L_{\uub_{2 1}, \nub_{2 1}}        & \Omega^{n_{2} + \cdots + n_{d}}                                          & 0         & \cdots            & 0 \\
\Omega^{n_{1} + \cdots + n_{d}} L_{\uub_{3 1}, \nub_{3 1}}        & \Omega^{n_{2} + \cdots + n_{d}} L_{\uub_{3 2}, \nub_{3 2}}         & \ddots &                      & \vdots \\
\vdots                                                                          &  \vdots                                                                          & \ddots & \Omega^{n_{d}} & 0 \\
\Omega^{n_{1} + \cdots + n_{d}} L_{\uub_{d+1, 1}, \nub_{d+1, 1}} & \Omega^{n_{2} + \cdots + n_{d}} L_{\uub_{d+1, 2}, \nub_{d+1, 2}} & \cdots &
                                                                                                                                                   \Omega^{n_{d}} L_{\uub_{d+1, d}, \nub_{d+1, d}} & 1 \\
\end{bmatrix},$$
where the notations $\nub_{i j}$ and $\uub_{i j}$ are defined in Definition \ref{def_index}.
These satisfy the Frobenius difference equations (\ref{Fr-eq}).
Hence $\Psi$ is a rigid analytic trivialization of $\Phi$.
Let $M$ be the pre-$t$-motive defined by $\Phi$.
By Theorem \ref{Psi-M}, we have an isomorphism $G_{\Psi} \to G_{M}$ and
$$\tdeg_{\Kb} \Kb(\pit, L_{\uub_{i j}, \nub_{i j}}(\theta) | (i, j) \in I_{d}) = \dim G_{\Psi}.$$
Thus when $\uub = H(\nub)$ (resp. $\uub = \aub \in \Kb^{d}$ with $|\alpha_{i}|_{\infty} < \niq$ for each $i$),
the multizeta values $\zeta(\nub_{i j})$ (resp. the Carlitz multiple polylogarithms $\Li_{\nub_{i j}}(\aub_{i j})$)
appear as periods of the pre-$t$-motive $M$.
By the definition of $G_{\Psi}$, we also have the inclusion
$$G_{\Psi} \subset
\left \{ \begin{bmatrix} a^{n_{1} + \cdots + n_{d}} & & & \\
x_{2 1} & a^{n_{2} + \cdots + n_{d}} & & \\
\vdots & \ddots & \ddots & \\
x_{d+1, 1} & \cdots & x_{d+1, d} & 1 \\ \end{bmatrix} \right \}.$$
We can calculate $\Psit$ explicitly as
\begin{eqnarray*}
\Psit_{i j} = (\Omega^{-1} \otimes \Omega)^{n_{i} + \cdots + n_{d}} \sum_{s = j}^{i} \sum_{r=0}^{i - s} (-1)^{r}
\sum_{\substack{s = i_{0} < i_{1} < \cdots \\ < i_{r-1} < i_{r} = i}} L_{i_{1} i_{0}} \cdots L_{i_{r} i_{r-1}} \otimes \Omega^{n_{j} + \cdots + n_{i-1}} L_{s j}
\end{eqnarray*}
for each $(i, j) \in I_{d}$, where we write $L_{k \ell} := L_{\uub_{k \ell}, \nub_{k \ell}}$.
\end{exam}

\section{Algebraic independence}\label{proof}
In this section, we prove Theorems \ref{main_zeta} and \ref{main_li}.
For square matrices $A$ and $B$, we denote by $A \oplus B$ the diagonal block matrix made of $A$ and $B$.
We use the letters $a$ and $x_{ij}$'s as coordinate variables of algebraic groups.
In our proofs, our purpose is to show that the dimension of the algebraic group in question is as maximal as possible,
and so we always work on the $\Fqtb$-valued points without studying the reduced/non-reduced structures,
where $\Fqtb$ is a fixed algebraic closure of $\Fq(t)$.
So for an algebraic group $G$ over $\Fq(t)$, when it is clear from the contents, without confusion we still denote by $G$ the $\Fqtb$-valued points of $G$.

Before starting the proof of Theorems \ref{main_zeta} and \ref{main_li}, we state several algebraic independence results concerning the case of depth one.
Papanikolas, Chang and Yu proved the following theorem which states a criterion of the algebraic independence of MZVs and CMPLs at algebraic points of depth one.
Note that they discussed only the case where $u_{j} = \alpha_{j} \in \Kb$ with $|\alpha_{j}|_{\infty} < \nq$, but their arguments work also for any $u_{j} \in \Kb[t]$
with $\ujval < \nq$.
\begin{theo}[{\cite[Theorem 6.3.2]{Papa}}, {\cite[Theorem 3.1]{ChYu}}]\label{lin-alg}
Let $n \geq 1$ be a positive integer and $u_{1}, \dots, u_{r} \in \Kb[t]$ polynomials with $\ujval < \nq$ for each $j$.
If $\pit^{n}, L_{u_{1}, n}(\theta), \dots, L_{u_{r}, n}(\theta)$ are linearly independent over $K$,
then they are algebraically independent over $\Kb$.
\end{theo}
Thus $\pit$ and $\zeta(n)$ (or $\Li_{n}(\alpha)$) are algebraically independent over $\Kb$
for each ``odd'' integer $n \geq 1$ and $\alpha \in K^{\times}$ with $|\alpha|_{\infty} < \nq$,
because $\pit^{n} \not\in \Kinf$ and $\zeta(n), \Li_{n}(\alpha) \in \Kinf^{\times}$ for such $n$ and $\alpha$.

\begin{theo}\label{alg-dep1}
Let $n_{1}, \dots, n_{d} \geq 1$ be positive integers such that $n_{i} / n_{j}$ is not an integral power of $p$ for each $i \neq j$.
For each $i$, we take polynomials $u_{i 1}, \dots, u_{i r_{i}} \in \Kb[t]$ with $|\!|u_{i j}|\!|_{\infty} < \niq$ for $j = 1, \dots, r_{i}$.
If $\pit^{n_{i}}, L_{u_{i 1}, n_{i}}(\theta), \dots, L_{u_{i r_{i}}, n_{i}}(\theta)$ are linearly independent over $K$ for each $i$,
then the $1 + \sum_{i = 1}^{d} r_{i}$ elements $\{ \pit, L_{u_{i j}, n_{i}}(\theta) | 1 \leq i \leq d, \ 1 \leq j \leq r_{i} \}$
are algebraically independent over $\Kb$.
\end{theo}
This is almost proved in \cite{ChYu}.
For the sake of completeness, we give a proof of it in \ref{app_proof}.

The next theorem is the main result in this paper.
Clearly, Theorems \ref{main_zeta} and \ref{main_li} follow from Theorems \ref{lin-alg}, \ref{alg-dep1} and \ref{main}.
Recall that $I_{d}$ is the set defined in Definition \ref{def_index}.
The notations $\nub_{i j}$ and $\uub_{i j}$ are also defined there.
\begin{theo}\label{main}
Let $\nub = (n_{1}, \dots, n_{d})$ be an index and $\uub = (u_{1}, \dots, u_{d}) \in (\Kb[t])^{d}$ a $d$-tuple of polynomials
such that $\uival < \niq$ for each $i$.
If $\pit, L_{u_{1}, n_{1}}(\theta), \dots, L_{u_{d}, n_{d}}(\theta)$ are algebraically independent over $\Kb$,
then we have
$$\tdeg_{\Kb} \Kb(\pit, L_{\uub_{i j}, \nub_{i j}}(\theta) | (i, j) \in I_{d}) = 1 + \# I_{d} = 1 + \frac{d(d+1)}{2}.$$
\end{theo}

To prove Theorems \ref{alg-dep1} and \ref{main}, we use the following lemma.
This lemma is clear, but very useful.
\begin{lemm}\label{V=0}
Let $V \subset \Ga^{r}$ be an algebraic subgroup of dimension zero.
Let $m_{1}, \dots, m_{r} \in \Z$ be non-zero integers.
Assume that $V$ is stable under the $\Gm$-action on $\Ga^{r}$ defined by
$$a. (x_{1}, \dots, x_{r}) = (a^{m_{1}} x_{1}, \dots, a^{m_{r}} x_{r}) \ \ (a \in \Gm, (x_{i}) \in \Ga^{r}).$$
Then $V(\Fqtb)$ is trivial.
\end{lemm}

\begin{proof}[Proof of Theorem \ref{main}]
In this proof, $(i, j)$ and $(k, \ell)$ are always assumed to be elements of the totally ordered set $I_{d}$.
Let $\Phi$ and $\Psi$ be the $(d+1) \times (d+1)$-matrices defined in Example \ref{ex-L}.
These satisfy the Frobenius difference equations (\ref{Fr-eq}).
For $(k, \ell) \in I_{d}$, we define $(\dep(k, \ell) + 1) \times (\dep(k, \ell) + 1)$-matrices
$\Phi[k, \ell] = (\Phi[k, \ell]_{ij})$ and $\Psi[k, \ell] = (\Psi[k, \ell]_{ij})$ which are sub-matrices of $\Phi$ and $\Psi$,
where $\Phi[k, \ell]_{ij} = \Phi_{i+\ell-1, j+\ell-1}$ and $\Psi[k, \ell]_{ij} = \Psi_{i+\ell-1, j+\ell-1}$.
In particular, the lower left corner of $\Phi[k, \ell]$ (resp. $\Psi[k, \ell]$) is the $(k, \ell)$-th entry of $\Phi$ (resp. $\Psi$).
The following is the illustration of the relative positions of the matrices:
\begin{eqnarray*}
\makebox[0pt][c]{\raisebox{-21pt}[0pt][0pt]{ \ \ $(k, \ell)$} \raisebox{-19pt}[0pt][0pt]{\begin{xy} \ar (14, 2)\end{xy}}}
\begin{bmatrix}
\cdot & & & & \\
\circ & \cdot & & & \\
\cdashline{3-4}
\circ & \multicolumn{1}{c:}{\circ} & \cdot & \multicolumn{1}{c:}{} & \\
\circ & \multicolumn{1}{c:}{\circ} & \circ & \multicolumn{1}{c:}{\cdot} & \\
\cdashline{3-4}
\circ & \circ & \circ & \circ & \cdot
\end{bmatrix}
\makebox[0pt][c]{\raisebox{-5pt}[0pt][0pt]{ \ \ \ \ \ \ \ \ \ \ \ \ \ \ \ \ \ \ \ \ \ \ \ \ \begin{xy} \ar (-7, 0) \end{xy}}
\raisebox{-8pt}[0pt][0pt]{$\Phi[k, \ell]$ (resp. $\Psi[k, \ell]$)}}
\makebox[0pt][c]{\raisebox{20pt}[0pt][0pt]{ \ \ \ \ \ \ \ \ \ \ \ \ \ \ \ \ \ \ \ \begin{xy} \ar (-7, 0) \end{xy}}
\raisebox{17pt}[0pt][0pt]{$\Phi$ (resp. $\Psi$)}}
\end{eqnarray*}
Let $M[k, \ell]$ be the pre-$t$-motive defined by $\Phi[k, \ell]$
and $G(k, \ell)$ the fundamental group of the pre-$t$-motive
$$M(k, \ell) := C \oplus \bigoplus_{\substack{\dep(i, j) \geq \dep(k, \ell) - 1 \\ (i, j) \leq (k, \ell)}} M[i, j],$$
where $C$ is the Carlitz pre-$t$-motive (see Example \ref{ex-C}).
The closed circles in the following illustration is the range in which $(i, j)$ runs over in the above direct sum:
\begin{eqnarray*}
\makebox[0pt][c]{\raisebox{-27pt}[0pt][0pt]{ \ \ $(k, \ell)$} \raisebox{-25pt}[0pt][0pt]{\begin{xy} \ar (10, 2)\end{xy} \ \ \ \ }}
\begin{bmatrix}
& & & & & & \\
\cdashline{1-1}
\multicolumn{1}{c:}{\circ} & & & & & & \\
\cdashline{2-2}
\bullet & \multicolumn{1}{c:}{\circ} & & & & & \\
\cdashline{3-3}
\bullet & \bullet & \multicolumn{1}{c:}{\circ} & & & & \\
\cdashline{4-4}
\circ & \bullet & \bullet & \multicolumn{1}{c:}{\circ} & & & \\
\cdashline{5-5}
\circ & \circ & \circ & \bullet & \multicolumn{1}{c:}{\circ} & &
\end{bmatrix}
\end{eqnarray*}
We identify $G(k, \ell)$ with the algebraic group defined by $\begin{bmatrix} \Omega \end{bmatrix} \oplus \bigoplus_{(i, j)} \Psi[i,j]$ as in Theorem \ref{Psi-M}.
Then we have the inclusion
$$G(k, \ell) \subset \left \{ \begin{bmatrix} a \end{bmatrix} \oplus \bigoplus_{\substack{\dep(i, j) \geq \dep(k, \ell) - 1 \\ (i, j) \leq (k, \ell)}}
\begin{bmatrix} a^{n_{j} + \cdots + n_{d}} & & & \\
x_{j+1,j} & a^{n_{j+1} + \cdots + n_{d}} & & \\
\vdots & \ddots & \ddots & \\
x_{ij} & \cdots & x_{i,i-1} & a^{n_{i} + \cdots + n_{d}} \\ \end{bmatrix} \right \}$$
for each $(k, \ell)$.
Note that  some different entries/coordinates of different block matrices may be the same and denoted by same letters;
this means that
for $(i, j), (i', j')$ and $r, r', s, s'$ with $1 \leq s < r \leq \dep(i, j) + 1$ and $1 \leq s' < r' \leq \dep(i', j') + 1$,
if $(r + j - 1, s + j - 1) = (r' + j' - 1, s' + j' - 1)$,
then the $(r, s)$-th entry of the $(i, j)$-th component matrix and the $(r', s')$-th entry of the $(i', j')$-th component matrix are the same
and they are denoted by $x_{r + j - 1, s + j - 1}$.
In fact, since $\Psi$ is a lower triangular matrix, the $(r, s)$-th entry of $\widetilde{\Psi[i,j]}$
is equal to the $(r + j - 1, s + j - 1)$-th entry of $\Psit$ (for the explicit description of $\Psit$, see Example \ref{ex-L}).
Thus if $(r + j - 1, s + j - 1) = (r' + j' - 1, s' + j' - 1)$,
then the $(r, s)$-th entry of $\widetilde{\Psi[i, j]}$ and the $(r', s')$-th entry of $\widetilde{\Psi[i', j']}$ coincide.
Therefore the corresponding positions in the algebraic group $G(k, \ell)$ are the same.

By Theorem \ref{Psi-M}, it suffices to show that the above inclusion is actually an equality for each $(k, \ell)$.
We prove this by induction on $(k, \ell) \in I_{d}$ via the total order ``$\leq$''.

By the assumption, this is true for $(2, 1) \leq (k, \ell) \leq (d+1, d)$, these are the depth one cases.
Let $(k, \ell) \geq (3, 1)$ (this means $\dep(k, \ell) \geq 2$) and assume that the inclusion is an equality for $(k', \ell')$ the greatest element of
$\{ (i, j) \in I_{d} | (i, j) < (k, \ell) \}$,
which means that $(k', \ell') = (k - 1, \ell - 1)$ if $\ell \neq 1$ and $(k', \ell') = (d + 1, d + 3 - k)$ if $\ell = 1$.
By definition, $M(k', \ell')$ is a subobject of $M(k, \ell)$ and $C$ is a subobject of $M(k, \ell)$ and $M(k', \ell')$.
By Tannakian duality, we have surjections
$\psi \colon G(k, \ell) \to G(k', \ell')$, $\pi \colon G(k, \ell) \to \Gm$ and $\pi' \colon G(k', \ell') \to \Gm$,
where we identify $G_{C}$ with $\Gm$.
These are projection maps.
More precisely, $\pi$ and $\pi'$ map the matrices of the above forms to $a$
and $\psi$ maps to the same matrices with the $(k, \ell)$-th component matrices removed.
These follow from the description of the map (\ref{map_G}).
The arguments are as same as in \cite[$\S$6.2.2]{Papa}, \cite[$\S$4.3]{ChYu} and \cite[Remark 2.3.2]{CPY}.
We set $V := \Ker \pi$ and $V' := \Ker \pi'$ to be the unipotent radicals of $G(k, \ell)$ and $G(k', \ell')$.
Then we have the following diagram
\[\xymatrix{
1 \ar[r] & V \ar[r] \ar[d]^{\psi|_{V}} & G(k, \ell) \ar[r]^{\pi} \ar@{>>}[d]^{\psi} & \Gm \ar[r] \ar@{=}[d] & 1 \\
1 \ar[r] & V' \ar[r] & G(k', \ell') \ar[r]^{\pi'} & \Gm \ar[r] & 1 \\
}\]
which is commutative and whose rows are exact.

It is clear that $\psi|_{V}$ is surjective.
Since $V$ is non-commutative, the $G(k, \ell)$-action $A . X := A^{-1} X A$ on $V$ ($X \in V$, $A \in G(k, \ell)$)
depends not only on $\pi(A)$ but also on the other entries of $A$.
Note that the coordinate variable $x_{k \ell}$ of $G(k, \ell)$ is the only coordinate variable which does not appear as a coordinate variable of $G(k', \ell')$.
Thus we know that $\dim G(k', \ell') \leq \dim G(k, \ell) \leq \dim G(k', \ell') + 1$.
This also follows from Theorem \ref{Psi-M} (2).
It suffices to show that the second inequality is an equality.

Now, assume that $\dim G(k, \ell) = \dim G(k', \ell')$.
Then $\dim \Ker (\psi|_{V}) = 0$.
It is clear that $\Ker (\psi|_{V})$ is a normal subgroup of $G(k, \ell)$ and
$A.x_{k \ell} = \pi(A)^{n_{\ell} + \cdots + n_{d}} x_{k \ell}$
for each $x_{k \ell} \in \Ker (\psi|_{V})$ and $A \in G(k, \ell)$,
where we identify $\Ker (\psi|_{V}) \subset \Ga$ by means of the coordinate $x_{k \ell}$.
By Lemma \ref{V=0} we have that $\Ker (\psi|_{V})$ is trivial.
We take any elements
$$X = \begin{bmatrix} 1 \end{bmatrix} \oplus \bigoplus_{\substack{\dep(i, j) \geq \dep(k, \ell) - 1 \\ (i, j) \leq (k, \ell)}}
\begin{bmatrix} 1 & & & \\
x_{j+1, j} & 1 & & \\
\vdots & \ddots & \ddots & \\
x_{i j} & \cdots & x_{i, i-1} & 1 \\ \end{bmatrix} \in V$$
and
\begin{eqnarray*}
A = \begin{bmatrix} 1 \end{bmatrix} &\oplus& \bigoplus_{\dep(i, j) = \dep(k, \ell) - 1}
\begin{bmatrix} 1 & & & \\
& 1 & & \\
& \mbox{\smash{\huge 0}} & \ddots & \\
a_{ij} & & & 1 \\ \end{bmatrix} \\
&\oplus& \bigoplus_{\substack{\dep(i, j) = \dep(k, \ell) \\ (i, j) \leq (k, \ell)}}
\begin{bmatrix} 1 & & & & & \\
& 1 & & & \\
& & \ddots & & \\
a_{i-1,j} & \mbox{\smash{\huge 0}} & & \ddots & & \\
a_{ij} & a_{i,j+1} & &  & 1 \\ \end{bmatrix} \in V,
\end{eqnarray*}
where we can take any $x_{ij} \in \Fqtb$ (resp. any $a_{ij} \in \Fqtb$) for each $(i, j) \in I_{d}$ such that $(i, j) \neq (k, \ell)$
(resp. $\dep(i, j) \geq \dep(k, \ell) - 1$ and $(i, j) < (k, \ell)$)
by the assumption on $(k', \ell')$ and the surjectivity of $\psi|_{V}$.
Then $X^{-1} (A^{-1} X A)$ is equal to
\begin{eqnarray*}
\begin{bmatrix} 1 \end{bmatrix} &\oplus& \bigoplus_{\dep(i, j) = \dep(k, \ell) - 1}
\begin{bmatrix} 1 & & \\
& \ddots & \\
& & 1 \\ \end{bmatrix} \\
&\oplus& \bigoplus_{\substack{\dep(i, j) = \dep(k, \ell) \\ (i, j) \leq (k, \ell)}}
\begin{bmatrix} 1 & & & & & \\
& 1 & & & \\
& & \ddots & & \\
& \mbox{\smash{\huge 0}} & & \ddots & & \\
a_{i-1,j} x_{i,i-1} - a_{i,j+1} x_{j+1,j} & & &  & 1 \\ \end{bmatrix}.
\end{eqnarray*}
Now we take $a_{i j} = 0$ for $(k - \ell, 1) \leq (i,j) < (k, \ell + 1)$ and $a_{k, \ell + 1} = 1$.
Then we see that $X^{-1} (A^{-1} X A) \in \Ker (\psi|_{V}) = \{ 0 \}$ and so we have $x_{\ell+1, \ell} = 0$.
Since $(\ell + 1, \ell) \neq (k, \ell)$, this is a contradiction.
Therefore we have $\dim G(k, \ell) = \dim G(k', \ell') + 1$.
\end{proof}

\appendix
\def\thesection{Appendix \Alph{section}}
\section{}\label{app_proof}
In this appendix, we prove Theorem \ref{alg-dep1}.
In \cite{ChYu}, they treated some special case, but their proof works also for any $n_{i}$ not divisible by $p$ and any $\alpha_{ij} \in \Kb[t]$ with $|\!|\alpha_{i j}|\!|_{\infty} < \niq$.
By a slight modification of their proof, we can weaken the condition on $n_{i}$'s as in our statement.
Note that when $\alpha_{ij} \in \Kb$ or $\alpha_{ij} = H_{n_{i}-1}$, we can reduce Theorem \ref{alg-dep1} to the case where $n_{i}$ is not divisible by $p$,
and we do not need the following proof in such cases.
As in Section \ref{proof}, for an algebraic group $G$ over $\Fq(t)$, when it is clear from the contents, without confusion we still denote by $G$ the $\Fqtb$-valued points of $G$.

\begin{proof}[Proof of Theorem \ref{alg-dep1}]
We set $I := \{ (i, j) \in \Z^{2} | 1 \leq i \leq d, \ 1 \leq j \leq r_{i} \}$.
In this proof, $(i, j)$ and $(k, \ell)$ are always assumed to be elements of $I$.
We define an order on $I$ by the lexicographic order;
this means $(i, j) \leq (k, \ell)$ if and only if $i = k$ and $j \leq \ell$, or $i < k$.
For $(k, \ell) \in I$, we define $2 \times 2$-matrices
$$\Phi[k, \ell] :=
\begin{bmatrix} (t-\theta)^{n_{k}} & 0 \\
\alpha_{k \ell}^{(-1)} (t-\theta)^{n_{k}} & 1 \end{bmatrix}
\ \ \ \mathrm{and} \ \ \
\Psi[k, \ell] :=
\begin{bmatrix} \Omega^{n_{k}} & 0 \\
\Omega^{n_{k}} L_{\alpha_{k \ell}, n_{k}} & 1 \end{bmatrix}.$$
Then they satisfy the Frobenius difference equations $\Psi[k, \ell]^{(-1)} = \Phi[k, \ell] \Psi[k, \ell]$.
Let $M[k, \ell]$ be the pre-$t$-motive defined by $\Phi[k, \ell]$
and $G(k, \ell)$ (resp. $G_{k}(\ell)$) the fundamental group of the pre-$t$-motive
$$M(k, \ell) := C \oplus \bigoplus_{(i, j) \leq (k, \ell)} M[i, j] \ \ \left ( \mathrm{resp. \ } M_{k}(\ell) := C \oplus \bigoplus_{j \leq \ell} M[k, j] \right ).$$
We identify $G(k, \ell)$ (resp. $G_{k}(\ell)$) with the algebraic group defined by $\begin{bmatrix} \Omega \end{bmatrix} \oplus \bigoplus_{(i, j)} \Psi[i,j]$
(resp. $\begin{bmatrix} \Omega \end{bmatrix} \oplus \bigoplus_{j} \Psi[k,j]$) as in Theorem \ref{Psi-M}.
Then we have the inclusion (resp. equality)
$$G(k, \ell) \subset \left \{ \begin{bmatrix} a \end{bmatrix} \oplus \bigoplus_{(i, j) \leq (k, \ell)}
\begin{bmatrix} a^{n_{i}} & 0 \\
x_{i j} & 1 \end{bmatrix} \right \} \ \ \left ( \mathrm{resp. \ } 
G_{k}(\ell) = \left \{ \begin{bmatrix} a \end{bmatrix} \oplus \bigoplus_{j \leq \ell}
\begin{bmatrix} a^{n_{k}} & 0 \\
x_{k j} & 1 \end{bmatrix} \right \} \right )$$
for each $(k, \ell)$.
By Theorem \ref{Psi-M}, it suffices to show that the above inclusion is actually an equality for each $(k, \ell)$.
We prove this by induction on $(k, \ell) \in I$ via the total order ``$\leq$''.

By the assumption, this is true for $(1, 1) \leq (k, \ell) \leq (1, r_{1})$.
Let $(k, \ell) \geq (2, 1)$ and assume that the inclusion is an equality for $(k', \ell')$ the greatest element of $\{ (i, j) \in I | (i, j) < (k, \ell) \}$.
Thus $(k', \ell') = (k, \ell - 1)$ if $\ell \neq 1$ and $(k', \ell') = (k-1, r_{k-1})$ if $\ell = 1$.
By definition, $M(k', \ell')$ and $M_{k}(\ell)$ are subobjects of $M(k, \ell)$ and $C$ is a subobject of $M(k, \ell)$, $M(k', \ell')$ and $M_{k}(\ell)$.
By the Tannakian duality, we have surjections
$\psi \colon G(k, \ell) \to G(k', \ell')$, $\psi_{k} \colon G(k, \ell) \to G_{k}(\ell)$, $\pi \colon G(k, \ell) \to \Gm$, $\pi' \colon G(k', \ell') \to \Gm$
and $\pi'' \colon G_{k}(\ell) \to \Gm$, where we identify $G_{C}$ with $\Gm$.
The projections $\pi$, $\pi'$ and $\pi''$ map the matrices of the above forms to $a$
and $\psi$ (resp. $\psi_{k}$) maps to the same matrices with the $(k, \ell)$-th component matrices (resp. all $(i, j)$-th component matrices ($i \neq k$)) removed.
We set $V := \Ker \pi$, $V' := \Ker \pi'$ and $V'' := \Ker \pi''$ to be the unipotent radicals of $G(k, \ell)$, $G(k', \ell')$ and $G_{k}(\ell)$.
Then we have the following diagram
\[\xymatrix{
1 \ar[r] & V'' \ar[r] & G_{k}(\ell) \ar[r]^{\pi''} & \Gm \ar[r] & 1 \\
1 \ar[r] & V \ar[r] \ar[d]^{\psi|_{V}} \ar[u]_{\psi_{k}|_{V}} & G(k, \ell) \ar[r]^{\pi} \ar@{>>}[d]^{\psi} \ar@{>>}[u]_{\psi_{k}} & \Gm \ar[r] \ar@{=}[d] \ar@{=}[u] & 1 \\
1 \ar[r] & V' \ar[r] & G(k', \ell') \ar[r]^{\pi'} & \Gm \ar[r] & 1 \\
}\]
which is commutative and whose rows are exact.

It is clear that $\psi|_{V}$ is surjective.
Note that the coordinate variable $x_{k \ell}$ of $G(k, \ell)$ is the only coordinate variable which does not appear as a coordinate variable of $G(k', \ell')$.
Thus we know that $\dim G(k', \ell') \leq \dim G(k, \ell) \leq \dim G(k', \ell') + 1$.
This also follows from Theorem \ref{Psi-M} (2).
It suffices to show that the second inequality is an equality.

Now, assume that $\dim G(k, \ell) = \dim G(k', \ell')$.
Then $\dim \Ker (\psi|_{V}) = 0$.
We identify
$V \subset \prod_{(i, j) \leq (k, \ell)} \Ga$, $V' = \prod_{(i, j) < (k, \ell)} \Ga$ and $V'' = \prod_{j \leq \ell} \Ga$ by means of the coordinates $x_{ij}$.
The $\Gm$-action on $V$ (resp. $V'$, resp. $V''$) defined by $a.X := \widetilde{a}^{-1} X \widetilde{a}$,
where $\widetilde{a} \in G(k, \ell)$ (resp. $G(k', \ell')$, resp. $G_{k}(\ell)$) is a lift of $a \in \Gm$,
is described as $x_{ij} \mapsto a^{n_{i}} x_{ij}$ on each coordinate.
By Lemma \ref{V=0} we have $\Ker (\psi|_{V}) = 1$.
Thus the morphism $\psi|_{V}$ is bijective (but not necessary an isomorphism of varieties)
and we have the surjective map
\[\xymatrix{
\psi_{k}|_{V} \circ \psi|_{V}^{-1} \colon V' & V \ar[l]_{\ \ \ \ \ \ \ \ \ \ \sim} \ar@{>>}[r] & V''. \\
}\]

For each $(i,j) \neq (k, \ell)$, we set $V_{i j}$ (resp. $V'_{i j}$) to be the subvariety of $V$ (resp. $V'$) defined by $x_{i' j'} = 0$ for each $(i', j') \neq (i, j), (k, \ell)$.
Then $\psi|_{V_{ij}} \colon V_{ij} \to V'_{ij} = \Ga$ is a bijective $\Gm$-homomorphism.
Thus we have $\dim V_{ij} = 1$.
Hence the algebraic set\footnote{More precisely, the smooth algebraic group $(V_{ij} \otimes \Fqtb)_{\red}$ is defined by such polynomial.}
$V_{ij}$ is defined by a separable polynomial of the form $x_{k \ell}^{p^{e}} - \sum_{n=0}^{m} b_{n} x_{ij}^{p^{n}}$
for some $e, m \geq 0$ and $b_{n} \in \Fqtb$ (See \cite[Corollary 1.8]{Conr}).
Now we take $i \neq k$ and assume that the $\Gm$-homomorphism $\psi_{k}|_{V_{ij}} \circ \psi|_{V_{ij}}^{-1}$ is non-zero.
Then we can take $b_{m} \neq 0$ and we have $(\sum_{n} b_{n} (a^{n_{i}} x_{ij})^{p^{n}})^{p^{-e}} = a^{n_{k}} (\sum_{n} b_{n} x_{ij}^{p^{n}})^{p^{-e}}$ for each $a \in \Gm$.
By comparing the coefficients of $x_{ij}^{p^{m-e}}$, we have $n_{i} p^{m-e} = n_{k}$, which is a contradiction.
Thus we conclude that $\psi_{d}|_{V_{ij}} \circ \psi|_{V_{ij}}^{-1} = 0$.
Therefore we have $\psi_{k}|_{V}(\psi|_{V}^{-1}(\Ga^{\ell-1})) = V''$, whence a contradiction since $\dim V'' = \ell$.
\end{proof}

\end{document}